\documentclass[10pt,b5paper,dvips]{amsart}
\usepackage{latexsym,amsfonts,amssymb,amsmath,xypic}
\usepackage[usenames,dvipsnames]{xcolor}

\newtheorem{defin}[subsection]{Definition.}
\newtheorem{lemma}[subsection]{Lemma.}
\newtheorem{thm}[subsection]{Theorem.}

\newtheorem{cor}[subsection]{Corollary.}
\newcommand{\br}{\mbox{$\mathbb R$}}

\newcommand{\cc}{{\mathcal C}}

\newcommand{\cv}{{\mathcal V}}

\renewcommand{\phi}{\varphi}

\begin{document}
\title{The generic rank for  $A$--plannar structures.}
\author{Jaroslav Hrdina}

\address{Institute of Mathematics, Faculty of Mechanical Engineering,
Brno University of Technology, Czech Republic.}
\email{hrdina@fme.vutbr.cz}


\begin{abstract}The paper mostly collects  material  
on generic rank of $A$--modules with respect to differential geometric applications. 
Our research was motivated by geometry of 
$A$--structures. In particular, we discuss the case where $A$ is an unitary associative 
algebra not necessary with inversion.  Some of the examples are studied in detail.

\end{abstract}

\maketitle
  \section{Motivation}
Let us say a few words about our geometric motivation. 
Various concepts generalizing geodetics have been studied for
almost quaternionic and similar geometries. Also various
structures on manifolds are defined as smooth distribution in the
vector bundle $T^{\star}M \otimes TM$ of all endomorphisms of the
tangent bundle. Very well known are  two examples: almost complex
and almost quaternionic structures. Let as extract some
formal properties from these examples. Unless otherwise stated,
all manifolds are smooth and they have the dimension $m$.  Let
$\nabla$ be a linear connection and let $c: \mathbb{R} \rightarrow
M$ be a curve on $M$. Then there is a vector field  $\dot{c} : =
\frac{\partial c(t) }{\partial t}: \mathbb{R} \rightarrow TM$ along 
the curve $c$.
Classically, a curve $c$ is a geodesic  if and only
if its tangent vectors $\dot{c}(t)$ are parallely transported
 along $c(t)$.
Let $M$ be a smooth manifold equipped with a linear connection
$\nabla$ and let $F$ be an affinor on $M$. A curve $c$ is called
{\em $F$--planar curve} \index{$F$--planar curve} if there is its
parametrization $c : \mathbb{R} \rightarrow M$ satisfying the
condition
$$ \nabla_{\dot{c}} \dot{c} \in \langle \dot{c} , F(\dot{c}) \rangle.$$
 It is easy to see that geodesics are $F$--planar curves for all
affinors $F$, because of $\nabla_{\dot{c}} {\dot{c}} \in \langle
\dot{c} \rangle \subset \langle \dot{c},F(\dot{c}) \rangle$.
 The best known example is  an almost complex structure. 
We have to be careful about
the dimension of $M$. Let $M$ be a manifold of dimension two and
let $I$ be a complex structure. A curve $c$ is $F$--planar for
$F=I$ if and only if  $c$ is satisfying the
identity $\nabla_{\dot{c}} \dot{c} \in \langle \dot{c},I \dot{c}
\rangle \cong \mathbb{R}^2$, and any curve $c$ satisfy the
identity $ \nabla_{\dot{c}} \dot{c} \in \mathbb{R}^2$. In other
words any curve $c$ is $F$--planar on the manifold of dimension
two. The concept of $F$--planar curves makes sense for dimension
at least four.
Consider  almost hypercomplex structure $(I,J,K)$. 
The curve  $c: \mathbb{R}
\rightarrow M$  such that $\nabla_{\dot{c}} {\dot{c}} \in \langle
\dot{c}, I(\dot{c}), J(\dot{c}), K(\dot{c}) \rangle$ is called
{\em $4$--planar}. It is easy to see
that all geodesics are $4$--planar curves, because of
$\nabla_{\dot{c}} \dot{c} \in \langle \dot{c} \rangle \subset
\langle \dot{c} , I(\dot{c}), J(\dot{c}),K(\dot{c}) \rangle$ and
also all $F$--planar curves are 4--planar, if  $F\in \langle
I,J,K,E \rangle$. 
This simple consequence of standard behavior of the generators of a vector
subspace suggests the generalization of the planarity concept below.

\begin{defin} \label{12} {\em
Let $M$ be a smooth manifold of dimension $m$. Let $A$ be a smooth
$\ell$--rank $(\ell<m)$ vector subbundle in $ T^* M \otimes TM $,
such that the identity affinor $E=id_{TM}$ restricted to $T_x M$
belongs to $A_x \subset T^*_xM \otimes T_xM $ at each point $x \in
M$. We say that $M$ is equipped by $\ell$--rank {\em
$A$--structure.} \index{$A$--structure}}
\end{defin}

In Definition \ref{12}, the dimension of $M$ is higher than the
rank of $A$. This is not a restriction, because there are no
$A$--structures of rank $\ell$ higher than $m$. The possibility
$\ell=m$ is not interesting, because in this event every curve is
$A$--planar.

\begin{defin} {\em
For any tangent vector $X \in T_x M$ we shall write $A(X)$ for the
vector subspace
$$ A(X) = \lbrace F(X) | F \in A_x M \rbrace \subset T_x M$$
and we call $A(X)$ the {\em $A$--hull of the vector
$X$}\index{$A$--hull}. Similarly, the {\em $A$--hull of vector
field} will be subbundle in $TM$ obtained pointwise.}
\end{defin}

\section{The generic rank}

For every smooth parameterized curve $c: \br \rightarrow M$  we
write $\dot{c}$ and $A(\dot{c})$ for the tangent vector field and
its $A$--hull along the curve $c$.

\begin{defin} {\em
Let $M$ be a smooth manifold equipped with an $A$--structure and a
linear connection $\nabla$.  A smooth curve  $c: \mathbb{R} \rightarrow M $ 
 is told to be {\em $A$--planar} if 
$$ \nabla_{\dot{c}} \dot{c} \in A(\dot{c}).$$
}\end{defin}

Clearly, $A$--planarity means that the parallel transport of any
tangent vector to $c$ has to stay within the $A$--hull
$A (\dot{c})$ of the tangent vector field $\dot{c}$ along the
curve. 
\begin{defin}{\em
Let $(M,A)$ be a smooth manifold $M$ equipped with an $\ell$--rank
$A$--structure. We say that the $A$--structure has
\begin{enumerate}
\item {\em generic rank $\ell$}  \index{generic rank $\ell$} if
for each $x \in M$ the subset of vectors $(X,Y) \in T_xM \oplus
T_x M$, such that the $A$--hulls $A(X)$ and $A(Y)$ generate a
vector subspace $A(X) \oplus A(Y)$ of dimension $2 \ell$ is open
and dense in $T_xM \oplus
T_x M$. 
\item {\em weak generic rank $\ell$} \index{weak
generic rank $\ell$} if for each $x \in M$ the subset of vectors
$$ \mathcal{V} := \lbrace X \in T_x M | \operatorname{dim} A(X) =
\ell \rbrace
$$ is open and dense in $T_x M$.
\end{enumerate} }
\end{defin}

One immediately checks that any $A$--structure which has generic
rank $\ell$ has week generic rank $\ell$. Indeed, if $U\subset T_xM$ is an
open subset of vectors $X$ with $A(X)$ of dimension lower than $\ell$, then
$U \times U$ is an open subset with to low dimension, too.

\begin{lemma} \label{1aff}
Let $M$ be a smooth manifold of dimension at least two
and $F$ be an affinor such that $F
\neq q \cdot E$, $q \in \mathbb R$. Then the $A$--structure, where 
$A=\langle E,F\rangle$ has weak generic rank $2$.
\end{lemma}
\begin{proof}
 Consider
$A$--structure $A = \langle E,F \rangle $. The complement of
$\mathcal{V}$ consists of vectors  $X \in T_x M$ such that:
$$ X + a F(X) = 0, \: a \in \br,$$
i.e.  eigenspace of $F$. Dimension of $A$ is two and $F$ is not
multiple of the identity. Thus,
the union of eigenspaces of $F$ is closed or
trivial vector subspace of $T_x M$. Thus, the complement $\mathcal{V}$
is open and nontrivial, i.e. open and dense.
\end{proof}

There is only one possibility for the $A$--structures in the
lowest dimension one $A=\langle E \rangle$. The algebra $\langle E
\rangle$ is an algebra with inversion, such that $E \cdot E = E$.
For every $X \in T_x M$, $A(X)$ is the straight line containing $A$.

Finally, let us recall two important examples. 
The pair $(M,F)$, where $M$ is a smooth manifold and $F$ is an affinor on $M$
is called a complex structure if and only if 
$ F^2 = - E = -id_M $. An almost complex structure has generic rank two 
on all manifolds of dimension at least four, because of Lemma \ref{1aff}. 
The pair $(M, F)$ is  called a product structure on $M$ if and only
if $ F^2 = E $ and $F\ne E$. An almost product structure has 
generic rank two on all manifolds of dimension at least four, because of 
Lemma \ref{1aff}.

\section{The case where $A$ is an algebra}

\begin{lemma}
Every $A$--structure $(M,A)$ on a manifold $M$, $\dim M \geq \dim A$, where $A$ is an algebra with inversion,
 has weak generic rank $\dim A$.
\end{lemma}
 
\begin{proof}
Consider $X$ such that $X \notin \mathcal{V}$, therefore $\exists F \in A =\langle E,G\rangle, FX =0$, 
and $F^{-1} F X = 0$ implies $X=0$.
\end{proof}
  
\begin{thm} \label{lemma_det}
Let $M$ be an $A$--structure and 
let $X_1, \dots, X_m$ be a basis of $V:=T_x M$, i.e. $V$ is an $A$--module. 
Let $A$ be an $n$--dimensional $k$--algebra, where $n < m $.  
If there exists $X\in V$ such that $\dim(A(X))=n$ then 
the $A$--structure has weak generic rank $n$.
\end{thm}
\begin{proof}
We prove equivalent statement, $A$--module $V$ does not have a generic rank 
$\ell$ if and only if there is a vector $X\in V$,  such that 
for any vector $Y \in V$ there is 
an affinor $G_{Y}$, such that 
$G_Y (X - \epsilon Y) =0$, 
for small $\epsilon$. Therefore, for any vector $Y \in V$ 
there is an affinor $G_Y$ such that 
$$G_Y (X)= \epsilon G_Y ( Y) $$
for small $\epsilon$.
Hence, the affinor $\frac{1}{\sqrt{\epsilon}} G_Y$ maps 
$\frac{1}{\sqrt{\epsilon}}X$ to a vector
$G_Y (Y)$
and therefore, 
for any vector $Y \in V$ 
there is an affinor $H_Y$
such that $H_Y (\frac{1}{\sqrt{\epsilon}} X)= H_Y ( Y)$. In particular, 
there is an affinor $S$ such that  $S(\frac{1}{\sqrt{\epsilon}}X) = S(Y+\frac{1}{\sqrt{\epsilon}}X)$ 
 and therefore,
for any $Y\in V$ there is an affinor $S_Y$ such that $S_Y(Y)=0$.
\end{proof}

\begin{thm} \label{a1} \label{weak-lemma} \label{l416}
Let $(M,A)$ be a smooth manifold of dimension $m$ equipped with
$A$--structure of rank $\ell$, such that $2\ell \leq m$. If
$A_x$ is an algebra (i.e. for all $f,g \in A_x , \: fg:=
f\circ g \in A_x$) for all $x\in M$,
and $A$ has weak generic rank $\ell$ then
the structure has generic rank $\ell$.
\end{thm}
\begin{proof}
Since the $A$--structure has a weak generic rank $\ell$, there is
an open and dense subset $\cv\subset TM$ such that
$\operatorname{dim}A(X) = \ell$ for all $X\in \cv$.

Because $A$ is an algebra, for any $X,Z \in TM$,  $Z \in A (X)$ implies also
$A(Z)
\subset A(X)$, and moreover $ A(Z) = A(X) $ for all
$X, Z\in \mathcal{V}$ because of the dimension.
Thus, whenever there is a non--trivial vector $0\ne Z\in A(X)\cap A(Y)$, the
entire subspaces coincide, i.e. $A(X) = A(Y)$.

In particular, whenever $X,Y\in \cv$ and
the dimension of $A(X)+A(Y)$ is less then $2\ell$,
we know $A(X)= A(Y)$.

Let us consider a couple of vectors $(Y,Z)\in A(X)\oplus A(X)$ for some
$X\in \cv$.
Consider a vector $W \notin A(X)$. An open
neighbourhood $\mathcal U$ of $Y$ has to include $(Y + a W, Y)$ for
all sufficiently small
$a \in \mathbb{R}$. But if $Y + aW \in A(X)$ for some $a\ne 0$
then $W \in A(X)$ and
this is not true. Thus, for every couple of vectors in $A(X) \oplus A(X)$
and
 for its every open neighbourhood, we have found another couple
$(Y'=Y+aW,Z)$ for which the dimension of $A(Y')+A(Z)$ is $2\ell$. This
proves the density of the set of couples of vectors generating the maximal
dimension $2\ell$.

Of course, the requirement on the maximal dimension is an open
condition and the theorem is proved.
\end{proof}

\begin{cor} \label{l415}
Let $(M,A)$ be a smooth manifold with $A$--structure of rank $\ell$,
such that 
$2\ell \leq \operatorname{dim} M$. If  $A_x \subset
T_x^{\star}M \otimes T_x M $ is an algebra with inversion then $A$ has weak
generic rank. Moreover, if $\operatorname{dim}M\ge 2\ell$ than $A$ has
generic rank $\ell$.
\end{cor}

\begin{cor} \label{l415}
Let $(M,A)$ be a smooth manifold with $A$--structure of rank $\ell$,
such that  $2\ell \leq \operatorname{dim} M$ and $A$ is an algebra. 
If there exists $X\in T_xM$ such that $\dim(A(X))=n$ then 
the $A$--structure has generic rank $n$.
\end{cor}

\section{Remark on Frobenious algebras}
Let  $A$ be an algebra over $\br$ with basis $\{F_i \}$, where $i = 1, \dots, n$, 
$F_1:= E$,  with structure constants 
$$C_{ij}^k \in \br  \text{ (i.e. } F_i F_j = C^s_{ij} F_s \text{)}.$$ 
In particular,  one can easily see that 
$$F_i=F_1F_i=C_{1i}^s F_s = \delta_i^s F_s, \text{ i.e. } C_{1i}^s=\delta_i^s.$$
 We introduce  matrices 
 $$ \hat{C}_i = (C^k_{ji}), \: \: \: \hat{C}_i^* = (C^j_{ik}), $$
 where $j$ is a number of rows.
 Then the associativity condition can be written as
 $$\hat{C}_j \hat{C}_k = C_{jk}^s \hat{C}_s \text{ or } 
 \hat{C}_j^* \hat{C}_k^* = C_{jk}^s \hat{C}_s^* $$
 and unity can be written as $\hat{C}_1=E$.
A linear functional $\epsilon: A \to \br$ is determined 
by the choice of a $n$--dimensional vector 
$\lambda = (\lambda_1, \dots,\lambda_n)$.

Now, for $\epsilon (F_i) = \lambda_i $
and  for $F = \sum_{i=1}^n a_i F_i  \in A$ we can  see immediately 
that $\epsilon (F) = \sum_{i=1}^n a_i \lambda_i $ and 
finally $F_i F_j = C_{ij}^s F_s$. We prove that, 
if  $\lambda \in \br^n$
be a vector such that the matrix $G = (g_{ij})$ is regular, 
where $g_{ij} :=  C_{ij}^s \lambda_s$. The functional $\epsilon : A \to \br$,
such that 
$$\epsilon :\sum_{i=1}^n a_i F_i \mapsto \sum_{i=1}^n a_i \lambda_i $$
is a Frobenius form.

The formula for generic rank $n$ from the Theorem \ref{lemma_det}.
reads that if there exists 
$X \in V$ such that $\{F_i X \}$  are linearly independent  
then $V$   has a generic rank.  
On the other hand, if there 
exists $\lambda \in \br^n$ such that $\{\hat{C}_i \lambda \}$   
are linearly independent then  $A$ is a Frobenius algebra.
This indicates that these  properties lead to similar conditions. 

In other words,
if $A$ is an algebra over $\br$ and the matrices $\hat{C}_i$ are structural matrices of $A$, 
then there is $B$--module $\br^n$, where $B = \langle \hat{C}_1, \dots, \hat{C}_n  
\rangle$.
Therefore, the algebra 
$A$ is a Frobenius algebra if and only if the $B$--module $\br^n$ has  generic rank $n$.

\section{Examples}
One can apply these results to two big groups of geometric structures, Clifford algebras and 
distributions.
\subsection{ An almost Cliffordian manifolds} 
An almost Clifford and almost Cliffordian manifolds are $G$-structures based on the definition of
Clifford numbers.
An almost Clifford manifold based on $ \cc l (s,t)$ is given by a reduction of the structure group 
$GL(km, \mathbb R)$ to  $GL(m, {\mathcal O})$, where $k=2^{s+t}$,
 $m \in \mathbb N$ and $\mathcal O$ is an arbitrary Clifford algebra. An almost Cliffordian manifold   is given by a reduction of the structure group
to  $GL(m, \mathcal O) GL(1, \mathcal O)$. 
It is easy to see that an almost Cliffordian structure is an $A$--structure, where
$A$ is a Clifford algebra $\mathcal O$ because
the affinors in the form of $F_0, \dots, F_{\ell} \in A $
can be defined only locally.
In \cite{hv12} authors prove the following theorem.
\begin{thm}
Let $F_0,\dots F_k$ denote the $k+1$ elements of the matrix representation of Clifford algebra 
$\cc l(s,t)$. Then there exists a real vector $X$ such that the dimension of a linear span
$\langle F_i X | i=1,\dots,k\rangle$ equals to $k+1$.
\end{thm}

Finally, let $M$ be a Cliffordian manifold, i.e. let $(M,A)$ be a smooth manifold with 
$A =\cc l(s,t)$, such that  $2^{s+t+1} \leq \operatorname{dim} M$,  
then Cliffordian manifold has generic rank $2^{s+t}$.
For more information about  almost Cliffordian structures see  papers \cite{hv12,hv12b}

\subsection{Distributions} If $D,\bar{D}$ form a complete system of distributions (i.e. they are disjoint and $D + \bar{D} = TM$)
than there are two affinors  $P, \bar{P}$associate with them such that 
$$P^2 = P, \: \: \bar{P}^2 = \bar{P}, \: P \bar{P} = \bar{P} P =0 \text{ and } P + \bar P = E,$$
where $\text{rank } P= r$ and $\text{rank } \bar{P}= \bar{r}$.

 The representation of distributions by affinors can by extended to any complete system 
$D_i$ such that the affinors $P_i$ satisfy the properties
$$P_i^2 = P_i, \: \: P_i P_j =0 \text{ for } i\neq j, \text{ and } \sum_i P_i =E. $$

Considering the element $P= a_1 P_1+ \cdots+ a_n P_n \in A$, the 
matrix  
$$
\begin{pmatrix}
E P \\  
P_1 P \\ 
\vdots  \\
P_n P 
\end{pmatrix}
$$
is following
$$
\begin{pmatrix}
a_1 &0 & \cdots & 0 \\  
0   &a_2 & \cdots & 0 \\ 
\vdots &\vdots &\cdots &\vdots \\
0 &0\cdots &0 &a_n 
\end{pmatrix}$$ 
and therefore 
$\langle P_1, \dots, P_n \rangle$ has  weak generic rank $n$ by Lemma
\ref{lemma_det}.  
Finally, let $M$ be a manifold with complete system of distributions $D_1, \dots ,D_n$, i.e. 
let $(M,A)$ be a smooth manifold with  $A =\langle  P_1, \dots , P_n\rangle $, such that  
$2 n \leq \operatorname{dim} M$,  
then the $A$--structure has generic rank $n$.

\section*{Acknowledgment}
This work was supported by the European Regional Development Fund in the IT4Innovations Center of Excellence project CZ.1.05/1.1.00/02.0070.


\begin{thebibliography}{99}
\bibitem{hs06} J. Hrdina, J. Slov{\'a}k, 
{\em Generalized planar curves and quaternionic geometry}, 
Global analysis and geometry Vol. 29 (2006), pp. 349--360.
\bibitem{H11} J. Hrdina {\em J. Remarks on F- planar curves and their generalizations},
Banach Center Publications, Vol. 93, (2011), pp. 241-249.
\bibitem{h11b} J. Hrdina, {\em Notes on connections attached to $A$--structures}, Differential Geometry and Applications 29, Supl. 1 (2011), 91--97.
\bibitem{hv12} J. Hrdina, P. Va\v{s}{\'i}k, {\em Generalized geodesics on some almost Cliffordian geometries},
Balkan Journal of Geometry and Its Applications, Vol. 17, No. 2, (2012)
pp. 41-48.
\bibitem{h12} J. Hrdina, {\em Geometric properties of Frobenius algebras}. 
Journal of Physics: Conference Series vol 1. (2012), pp. 1-5. 
\bibitem{hv12b} J. Hrdina, P. Va\v{s}{\'i}k, {\em 
Geometry of almost Cliffordian manifolds: classes of subordinated connections}, 
preprint (arXiv:1205.6048)
\end{thebibliography}
\end{document}